\documentclass[a4paper,12pt,reqno]{amsart}
\usepackage{latexsym}
\usepackage{amsmath,amsthm,amssymb,amscd}
\usepackage{fullpage}
\usepackage{enumitem}
\usepackage{xcolor}
\usepackage{parskip}
\usepackage{mathtools}
\usepackage{thmtools}
\mathtoolsset{showonlyrefs}
\usepackage{csquotes}
\usepackage{comment}
\usepackage[activate={true,nocompatibility},final,tracking=true,kerning=true,spacing=true]{microtype}
\usepackage[pdfencoding=unicode,pdftex,bookmarks=true,bookmarksnumbered]{hyperref}
\definecolor{citecolour}{rgb}{0.0, 0.0, 0.8}
\colorlet{linkcolour}{green!50!black}
\hypersetup{colorlinks,breaklinks,
		linkcolor=linkcolour,citecolor=citecolour,
		filecolor=linkcolour, urlcolor=linkcolour}

\usepackage{txfonts,pxfonts,tikz} 
\usepackage[T1]{fontenc}

\newtheorem{prevtheorem}{Theorem}

\newtheorem{theorem}{Theorem}

\newtheorem{proposition}{Proposition}

\newtheorem{remark}{Remark}
\newtheorem{lemma}[theorem]{Lemma}
\newtheorem{corollary}{Corollary}
\newtheorem{conjecture}{Conjecture}

\DeclareMathOperator{\adj}{adj}

\newenvironment{proofofA}{{\bf {Proof of Theorem \ref{Thm:A}.} }}{\hfill $\blacksquare$ \\}
\newenvironment{proofofB}{{\bf {Proof of Theorem \ref{Thm:B}.} }}{\hfill $\blacksquare$ \\}

\newenvironment{proofofPropA}{{\bf {Proof of \autoref{Prop:Unc-A}.} }}{\hfill $\blacksquare$ \\}

\def\Z{\mathbf{Z}}
\def\C{\mathbf{C}}
\def\N{\mathbf{N}}
\def\Q{\mathbf{Q}}

\def\F{\mathbf{F}}
\def\supp{\mathrm{supp}}
\def\Supp{\mathrm{Supp}}

\usepackage{graphicx} 

\begin{document}

\title{ Chebotarev's theorem  for groups of order $pq$  and an uncertainty principle}

\author{Maria Loukaki}
\address{Department of Mathematics \& Applied Mathematics, University of Crete, Greece}
\email{mloukaki@uoc.gr}


\begin{abstract}
Let $p$ be a prime number and $\zeta_p$ a primitive $p$-th root of unity. Chebotarev's theorem states that every square submatrix of the $p \times p$ matrix $(\zeta_p^{ij})_{i,j=0}^{p-1}$ is non-singular. In this  paper we  prove the same for principal  submatrices of $(\zeta_n^{ij})_{i,j=0}^{n-1}$, when  $n=pr$ is the product of two distinct primes, and $p$ is a large enough prime that has order $r-1$ in $\mathbf{Z}_r^*$. As an application,  an uncertainty principle for cyclic groups of order $n$ is established when $n=pr$ as described above.
\end{abstract}

\maketitle
\section{Introduction}
For any integer $ n$, let $ \zeta_n$  denote a primitive $n$-th root of unity. For a prime $p$  and a positive integer $k$, we denote the field with $ p^k $ elements as $\F_{p^k}$.  The matrix $(\zeta_p^{ij})$ with $i,j \in \{ 0, 1, \cdots, p-1\}$ is a Vandermonde matrix with non-zero determinant. It was Chebotarev, in 1926, who first noticed that any square submatrix of the above inherits the same property (see \cite{Lenstra}). He proved 
\begin{theorem}\label{Thm:Chebotarev}
If $I, J \subseteq \{ 0, 1, 2, \cdots p-1 \}$  with $|I|=|J|$ then the matrix 
$(\zeta_p^{ij})_{i\in I, j \in J}$ has non-zero determinant.
\end{theorem}

After Chebotarev's initial proof, several others have emerged in the literature; Dieudonne \cite{Dieudonne}, Evans and Isaacs \cite{EvIs}, Goldstein et al. \cite{Gold},  Tao \cite{Tao}, Frenkel \cite{Frenkel},
 just to mention a few. Several approaches were used but  all require some  algebraic information. Chebotarev's original idea was to show that the 
  determinats of the matrices in question, which are  elements of the $p$-th cyclotomic field $\Q(\zeta_p)$ over $\Q$,  do not vanish in the $p$-adic completion $\Q_p(\zeta_p)$  of $\Q(\zeta_p)$. Tao \cite{Tao} constructs a polynomial in $n$ variables and differentiates it repeatedly to obtain an integer-coefficient polynomial $P$ where $P(1, 1, \cdots , 1)$ is not a multiple of $p$, contradicting his cyclotomic integer lemma. In addition he  proves the equivalence between Chebotarev's theorem and a Fourier uncertainty principle in $\Z_p$. The proof in  \cite{EvIs} employs homogeneous polynomials, and  a theorem by Mitchell \cite{Mitchell}, while  the method used in  \cite{Gold} is much more algebraic and also establish Tao's uncertainty principle in an algebraic context. Finally, Frenkel's method involves a single-variable polynomial in $\Z[\zeta_p][x]$ that reduces to a polynomial in $\Z_p[x]$ with an overly large root multiplicity.

 If $n$ is not a single prime, we still have the Vandermonde matrix $(\zeta_n^{ij})_{0 \leq i, j \leq n-1}$ which is nonsingular, but we can find examples of square submatrices of the original that are singular (see \cite{CL.24}). On the other hand, no counterexample of a singular principal submatrix when $n$ is a square-free integer is known.  (By definition, a principal submatrix  of $(\zeta_n^{ij})_{0 \leq i,j \leq n-1}$ is any submatrix $(\zeta_n^{ij})_{i,j \in I}$ for   some $I \subseteq \{0, 1, \cdots, n-1 \}$).  The following conjecture was first stated in \cite{Cabreli}, with a more general version presented in \cite{CL.24}.
 \begin{conjecture}\label{Conjecture}
 If $n$ is square-free then all principal submatrices of $(\zeta_n^{ij})$ with $i,j \in \{ 0, 1, \cdots, n-1\}$
 have nonzero determinant.     
 \end{conjecture}

In \cite{Cabreli}, the conjecture was confirmed for all $2 \times 2$ principal submatrices. Even more was actually proved:  for $n \geq 2$, all principal $2 \times 2$ submatrices of $(\zeta_n^{ij})_{0\leq i,j \leq n-1}$ are invertible if and only if $n$ is square-free (see Corollary 2.14 in \cite{Cabreli}).
  In addition to the above,  in \cite{CL.24} it was proved that, for $n >4$,  all principal  $3 \times 3$ submatrices of
 $(\zeta_n^{ij})_{0\leq i,j \leq n-1}$ are invertible if and only if $n$ is square free (see Theorem 1.1 in \cite{CL.24}).
These  results also imply that  any  $(n-2) \times (n-2)$ and any  $(n-3) \times (n-3)$ principal submatrix 
of $(\zeta_n^{ij})_{0\leq i,j \leq n-1}$ is non-singular, see Section \ref{Complementary}.

The truth of Conjecture \ref{Conjecture} has important implications for problems connected to bases of finite dimensional linear spaces. For instance, in \cite{Cabreli}, it is shown that two ordered bases of $\mathbf{C}^n$ are \textit{woven} (this means that we can drop some elements of one basis as long as we replace them with the corresponding elements of the other basis, and so keep the basis property) if and only if all principal minors of the change-of-basis matrix are nonsingular. When the two bases of $\mathbf{C}^n$ in consideration are the usual basis $\{e_1, \ldots, e_n\}$ and the Fourier basis $\{u_i = (1, \zeta_n^i, \zeta_n^{2i}, \ldots, \zeta_n^{(n-1)i}), i=0, 1, \ldots, n-1\}$ the validity of Conjecture \ref{Conjecture} for the dimension $n$ implies that for any subset $I$ of $\{0, 1, \ldots, n-1\}$ a vector $x \in \mathbf{C}^n$ can be reconstructed by the values $x_i$, $i \in I$, and $\widehat{x}_j$, $j \notin I$. Here $\widehat{x}$ is the Fourier Transform of the vector $x$ (in the cyclic group of order $n$) or, in other words, the coefficients of $x$ written in the Fourier basis $\{u_i\}$.

If $n=p r$  with $p,r$ distinct primes,     then $\Z_p \times \Z_r \cong \Z_n$ under the group  isomorphism that sends 
\begin{equation}\label{eq:isomo}
   \Z_p \times \Z_r \ni (i_p, \, i_r) \longrightarrow   i \in   \Z_n,
\end{equation}
 with   $i \equiv i_p \cdot r +i_r \cdot p  \pmod{ pr}$.  So $i \equiv i_p r \pmod p$ and $i \equiv i_r p \pmod r$. 
 
  [The map above  is easily seen to be a group homomorphism. Since the groups have the same cardinality, it suffices to show that its kernel is trivial to prove it is an isomorphism. If $(a, b) \in \Z_p \times \Z_r$ is such that $ra + bp \equiv 0 \pmod{pr}$, then $p \mid ra$ and $r \mid bp$. Since $p$ and $r$ are coprime, it follows that $p \mid a$ and $r \mid b$ and thus $(a,b) = (0,0)$ as desired.]

Assume now that a primitive $n$-th root $\zeta_n$ of $1$ is fixed. Using the isomorphism above, for every $i \in \Z_n$,  we have 
\begin{equation}\label{eq:zn^i}
\zeta_n^i = (\zeta_n^r)^{i_p} \cdot (\zeta_n^p)^{i_r}= \zeta_p^{i_p} \cdot \zeta_r^{i_r}, 
\end{equation}
where $\zeta_p = \zeta_n^r$  and $\zeta_r = \zeta_n^p$ are  primitive $p$-th and $r$-th roots of unity, respectively.
Note also that for every $i, j \in \Z_n$  we  get  
\[
\zeta_n^{ij}= \zeta_p^{ri_pj_p} \cdot \zeta_r^{pi_rj_r}.
\]
If $I \subseteq \Z_n$, (where $n = pr$), we define
\begin{equation}\label{eq:def}
I_r^k = \{ i \in I \mid i_r = k  \}
\end{equation}
for each $k = 0, 1, \ldots, r-1$. 
 Observe that $I = \dot\bigcup_{k}\,  I_r^k$. 
 
In this paper, we heavily use the approach introduced by P.E. Frenkel in \cite{Frenkel}  to partially address \autoref{Conjecture}.  Our first main theorem deals with  the case where \( n=2p \) for some odd prime \( p \).
\begin{prevtheorem}\label{Thm:A}
 Assume that $n=2p$ with $p$ an odd prime and  let   $I, J \subseteq \Z_{n}$ with $|I_2^k| = |J_2^k|$, for  $k=0,1$.  Then the matrix $(\zeta_{n}^{ij})_{i \in I, j\in J}$ has nonzero determinant.    
\end{prevtheorem}

As principal  submatrices  are exactly those with $I = J$ (in the notation of \autoref{Thm:A}) we get the Corollary that follows. We remark however,  that 
principal submatrices are indeed a special case, as for example in the group $\Z_2 \times \Z_5$ 
the subsets 
$I = \{ (0,1), (0,2), (0,3), (1,3)\}$ and $J=\{ (0,0), (0,1), (0,4), (1,1)\}$ satisfy $|I_2^0| = |J_2^0|=3$ and $|I_2^1| = |J_2^1|=1$. In addition, to describe the isomorphism \eqref{eq:isomo} in this case, note that any element $(a,b) \in \Z_2 \times \Z_5$ maps to $5a +2b \pmod{10}$ in $\Z_{10}$. Consequently, the sets $I$ and $J$ are mapped to the subsets $\{2,4,6,1\}$ and $\{0,2,8,7\}$ of $\Z_{10}$, respectively.
\begin{corollary}\label{Cor:n=2p}
If  $n=2p$,  where $p$ is an odd prime, then every  principal submatrix of $(\zeta_n^{ij})_{0 \leq i,j\leq n-1}$ has non-zero determinant.
\end{corollary}

In the general case where \( n = pr \) and \( p, r \) are distinct odd primes, we can establish an analogue of \autoref{Thm:A} only under two additional assumptions about \( p \) and \( r \). First, \( p \) must be a primitive element in \( \Z_r \) (that is, $p$ is a generator of $\Z_r^*$ or equivalently the order of $p \pmod r$ is $r-1$).  Second, \( p \) must exceed a constant \( \Gamma_r \) that depends on \( r \) (details about \( \Gamma_r \) are provided in \autoref{Sec:n=pr}). 
The constant $\Gamma_r$ is required in our theorem  because we are relying on  the following theorem of G. Zhang (Theorem A in \cite{Zhang}),   that is the analogue of Chebotarev's theorem for finite fields.   
\begin{theorem}(Zhang)\label{Zhang}
    Let $p, r$ be distinct odd primes with $p$ primitive  in  $\Z_r$ and $p > \Gamma_r$.  Suppose that  $\omega$ is an $r$-th  primitive root of unity in $\F_{p^{r-1}}$. Then all square submatrices of $(\omega^{ij})_{i,j= 0}^{r-1}$  have nonzero determinants. 
\end{theorem}

Our main theorem reads:
\begin{prevtheorem}
\label{Thm:B}
Let $ n = pr$, where $ p $ and $r$ are distinct odd primes such that $p$ is primitive in $\Z_r$ and $ p >\Gamma_r$.    If $I, J \leq \Z_{n}$ with $|I_r^k| = |J_r^k|$, for all $k=0,1, \cdots,  r-1$, then the matrix $(\zeta_{n}^{ij})_{i \in I, j\in J}$ has nonzero determinant.    
\end{prevtheorem}

As with \autoref{Cor:n=2p},\autoref{Thm:B} implies:
\begin{corollary}\label{Cor:n=pr}
    Assume  $ n = pr$, where $ p $ and $r$ are distinct odd primes such that $p$ is primitive in  $\mathbb{Z}_r$ and $ p >\Gamma_r$. Then every principal  submatrix of $(\zeta_n^{ij})_{0 \leq i,j\leq n-1}$ has non-zero determinant.
\end{corollary}

While the inductive proof of \autoref{Thm:B} extends to more than two primes, our application of Zhang's Theorem restricts us to the two-prime case. 
As this paper was nearing publication, we learned that Caragea et al. \cite{CaLeMaPf} had improved Zhang's Theorem by removing the primitivity condition and  verifying \autoref{Conjecture} for any number of sufficiently separated primes. Additionally, Emmrich and Kunis \cite{EmKu} also improved Zhang's Theorem by lifting the primitivity restriction and also significantly lowering the bounds on the field's characteristic.

In \cite{Tao}, T. Tao proved that Chebotarev's theorem is equivalent to an improved uncertainty principle for complex-valued functions that  asserts: 
\[
|\supp(f)| + |\supp(\widehat{f})| \geq p + 1, 
\]
for every non-zero complex function $f : \Z_p \to \C$, where $\widehat{f}$ is the Fourier transform of $f$ defined on the dual group $\widehat{\Z_p}$. A. Biró \cite{Biro} and R. Meshulam \cite{Meshulam} also independently established this lower bound on the sum of the supports of $f$ and $\widehat{f}$.

Assume now that $G = \mathbb{Z}_r \times \mathbb{Z}_m$ with $(r, m) = 1$, and let $\widehat{G}$ be its dual group. Since $\widehat{G}$ is isomorphic to $G$ we can also choose to view $\widehat{G}$ as the group $\mathbb{Z}_r\times\mathbb{Z}_m$. For $i=0, 1, \ldots, r-1$ we can now define both $G_i$ and $\widehat{G}_i$ to be the subset $\{i\}\times\mathbb{Z}_m$.

Using a similar argument to that of Tao in \cite{Tao}
(see \autoref{Sec:Uncertainty} for the details) we get the following uncertainty  principle:
\begin{proposition}\label{Prop:Unc-A}
Let $ p$ be an odd prime. If $f: G= \Z_2 \times \Z_p \to \C $ is a non-zero function  then for at least one of $i=0 $ or $ i=1 $, we have 
\[
|\supp(f) \cap G_i| + |\supp(\widehat{f} \, ) \cap \widehat{G}_i| \geq p + 1.
\]
\end{proposition}
and 
\begin{proposition}\label{Prop:Unc-B}
Let $p$ and $r$ be distinct odd primes, with $p$ being primitive in $\Z_r$ and $p > \Gamma_r$. If $f: G:= \Z_r \times \Z_p \to \C$ is a non-zero function, then for at least one $i \in \{0, 1, \ldots, r-1\}$, we have 
\[
|\supp(f) \cap G_i| + |\supp(\widehat{f}) \cap \widehat{G}_i| \geq p + 1.
\]
\end{proposition}

{\bf Acknowledgment.}
I thank J. Antoniadis for his support and valuable discussions, G. Pfander for introducing me to the problem, and the referees for their insightful comments.

\section{Preliminaries}\label{Sec:prelim}

We begin with a lemma whose proof relies on fundamental concepts from algebraic number theory, briefly outlined here. Assume that $K$ is an  algebraic number field of degree $[K:\Q]=n$ with $K/\Q$ a Galois extension.  If $R$ is  the ring of algebraic  integers of $K$ then  any prime $p \in \Z$ satisfies  
\[
pR = \prod_{i=1}^r P_i^e, 
\]
where $\{ P_i \}$ are  the distinct prime ideals of $R$ lying above the ideal $p\Z$, and  $e:= e(P_i/p)$ is the ramification index of $P_i$ in $K/\Q$, which is the same for every $i$ as the extension is Galois.  Additionally, the residual degree $f_i:= f(P_i/p)$ of $P_i$ in $K/\Q$ is the same constant $f$ for each $i=1, \ldots, r$, and satisfies $f = [R/P_i : \Z/p\Z]$. Furthermore,
\begin{equation}\label{eq:efr}
 n=efr.   
\end{equation}
Both,  the ramification index and the residual degree follow some transitivity rules, in the sense that if $L/K$ is  a Galois extension, $S$ the ring of algebraic integers of $L$ and $Q$  a prime ideal of $S$ above $P$ then 
\begin{align*}
   f(Q/p)&= f(Q/P)\cdot f(P/p), \\
   e(Q/p)&= e(Q/P)\cdot e(P/p).
\end{align*}
The transitivity holds more general for separable extensions but we only need it here for Galois extensions.  All the above,  and much more, can be found in any book of Algebraic Number Theory, see for example Chapter 11 in \cite{Ribe}.

Assume now that $n=pm $ with $m$ coprime to $p$ and let $K= \Q(\zeta_n)$  be a cyclotomic field with  $R= \Z[\zeta_n]$ its ring of algebraic integers. In this special case  the decomposition of $pR$ into prime ideals in $R$ is given as 
\[
pR=  \prod_{i=1}^r P_i^e
\]
where  $e= \phi(p)=p-1$ and  $r= \frac{\phi(m)}{h}$ with  $h$ being the order of $p$ in the multiplicative group $\Z_m^*$. Furthermore, $f(P_i/p) = h$ and the norm of the ideals $P_i$ equals $N(P_i)=p^h$,  for all $i=1, \cdots, r$. 
For the general theorem regarding the factorization of a rational  prime $p$ into prime ideals in the ring of integers $\Z[\zeta_t]$ of  a cyclotomic field $\Q(\zeta_t)$ with $p \mid t$,  see Theorems 8.7 and 8.8 in \cite{Mann}.

We are now ready to prove our first lemma which generalizes  Lemma 1 in \cite{Frenkel}. 
\begin{lemma}\label{Lem:MaximalIdeal}
Let $p$ be an odd prime and $n = pm$ for some integer $m$ such that $(p, m) = 1$. The ideal $\langle 1 - \zeta_p \rangle = (1 - \zeta_p)\mathbb{Z}[\zeta_n]$ is prime (and hence maximal) in $\mathbb{Z}[\zeta_n]$ if and only if the order of $p$ in $\mathbb{Z}_m^*$ is $\phi(m)$ and so $\mathbb{Z}_m^*$ is cyclic and  $p$ is one of its  generators. In this case, the quotient 
\[
\mathbb{Z}[\zeta_{n}]/\langle 1 - \zeta_p \rangle = \F_{p^{\phi(m)}}
\] 
is a finite field of characteristic $p$ and order $p^{\phi(m)}$. Moreover, $m$ can only be one of  $m=1,2,4$, $q^k$, or $2q^k$ for some odd prime $q$ and a positive integer  $k$.
\end{lemma}
\begin{proof}
Let $ L = \Q(\zeta_n) $ and $ K = \Q(\zeta_p) $ be cyclotomic fields, with \( S = \Z[\zeta_n] \) and \( R = \Z[\zeta_p] \) their  ring of algebraic integers. Clearly $[L:\Q]= \phi(n) = \phi(p) \phi(m)$,  $[K:\Q]= \phi(p)$ and thus $[L:K]= \phi(m)$. Furthermore, $L= K(\zeta_m)$ is also a Galois extension of $K$ of degree $\phi(m)$.

It is well known, see for example Section 11.3  Proposition N  in \cite{Ribe}, that the  ideal $P = ( 1 - \zeta_p ) R$ is a prime ideal of $R$, and it  is  the only ideal of $R$  above $p\Z$. In addition
\[
pR= P^{p-1}.
\]
Hence the ramification index $e(P/p)$ of $P$ in $K/\Q$ is $e(P/p)=p-1$ and its residual degree $f(P/p)= 1$.
Now, if $\{ Q_i\}_{i=1}^r$ is the set of prime ideals of $S$ lying above $p$ then 
\[
pS= \prod_{i=1}^r Q_i^{e(Q/p)}
\]
where $Q=Q_1$,  $e(Q/p)=p-1$ and $r= \frac{\phi(m)}{h}$ where $h$ is the order of $p$ in the multiplicative group $\Z_m^*$, by 	our preliminary remarks on the decomposition of a rational prime  into prime ideals in cyclotomic fields. In addition, $f(Q/p)=h$ and the norm of $Q_i$ is $N(Q_i)= p^h$. 

Observe now that   $\{Q_i\}_{i=1}^r$ is exactly  the set of prime ideals of $S$ above $P$ as well,  while  the extension $L/K$ is also Galois,  and thus 
\[
PS = \prod_{i=1}^r Q_i^{e(Q/P)}.
\]
We conclude that the ideal $PS= (1-\zeta_p)\Z[\zeta_n]$  is prime in $S= \Z[\zeta_n]$ if and only if $1=r =\frac{\phi(m)}{h}$. Hence the first part of the lemma follows. 

By the transitivity of the ramification index  we get  
\[
  p-1 = e(Q/p) = e(Q/P)\cdot e(P/p)= e(Q/P) \cdot (p-1).
\]
Hence $e(Q/P)=1$. 
Thus, in the case that  $r=1$ we get  $PS=Q$ and so 
\[
h=f(Q/p) = [S/Q:\Z/p\Z]= [
\mathbb{Z}[\zeta_{n}]/\langle 1 - \zeta_p \rangle :\Z_p].
\]
We conclude that $\mathbb{Z}[\zeta_{n}]/\langle 1 - \zeta_p \rangle $ is a finite field of order $p^h= p^{\phi(m)}$. 

The final part of the lemma states the well-known fact, first proved by Gauss, that the only integers 
$m$ for which $ \Z_m^* $ is cyclic are $1,2, 4, q^k $, or $ 2q^k $ for some odd prime $ q $ and integer $ k$. 
\end{proof}

The above lemma easily implies:

\begin{lemma} \label{Lem:z_p-primitive}
    Assume $p, r$ are distinct odd primes such that $p$ is primitive in  $\Z_r$. 
    Then \[
    \Z[\zeta_{pr}]/\langle 1- \zeta_p \rangle = \F_{p^{r-1}}
    \]
    and the image $\bar{\zeta_r}$  of $\zeta_r \in \Z[\zeta_{pr}]$ in  $\Z[\zeta_{pr}]/\langle 1- \zeta_p \rangle$  is also a primitive $r$-th root of unity in the field $\Z[\zeta_{pr}]/\langle 1- \zeta_p \rangle$.
\end{lemma}

\begin{proof}
We only need to show that $\bar{\zeta_r}$ has order $r$ in $\Z[\zeta_{pr}]/\langle 1- \zeta_p \rangle$.
Clearly $\bar{\zeta_r}^r= \bar{1}$  and suppose that 
$\bar{\zeta_r}^k= \bar{1}$ for some $0 < k < r$. 
Hence $\zeta_r^k-1 \equiv 0 \pmod{(1- \zeta_p)}$  and so 
$(1- \zeta_p) \mid (\zeta_r^k-1)$ in $Z[\zeta_{pr}]$.
If $L = \Q[\zeta_{pr}]$, $M= \Q[\zeta_r]$ and $K= \Q[\zeta_p]$ then $[L:M]= p-1$ and  $[L:K]= r-1$. Additionally, the norm \( N_{L/Q}(1- \zeta_p) = p^{r-1} \) by  \autoref{Lem:MaximalIdeal}, or we can compute it directly as follows:
\[
N_{L/Q}(1- \zeta_p) = N_{L/K}(N_{K/Q}(1 - \zeta_p)) = N_{L/K}\left(\prod_{i=0}^{p-1} (1-\zeta_p^i)\right) = N_{L/K}(\Phi_p(1)) = N_{L/K}(p) = p^{r-1}. 
\]
For $0 < k < r$, as $i$ ranges from $0$ to $r-1$, the set $\{ \zeta_r^{ik} \}$ covers all primitive $r$-th roots of unity. Hence 
\[N_{L/Q}(1- \zeta_r^k)= N_{L/M}(N_{M/Q}(1 - \zeta_r^k))=N_{L/M}\left(\prod_{i=0}^{r-1} (1-\zeta_r^{ik})\right)=
N_{L/M}(\Phi_r(1))= N_{L/M}(r)=r^{p-1}. 
\]
If $(1 - \zeta_p) \mid (\zeta_r^k - 1)$ in $\Z[\zeta_{pr}]$  we should have 
\[
p^{r-1}= N_{L/Q} (1-\zeta_p) \, \big| \,  N_{L/Q}(\zeta_r^k-1)= r^{p-1}.
\]
But this last division occurs in $\Z$, leading to an obvious contradiction and thus completing the proof of the lemma.
\end{proof}

We conclude the preliminaries of Algebraic Number Theory with two remarks.
\begin{remark}\label{remark1}
 For every odd prime $p$, the element $1 - \zeta_p$ in $\Z(\zeta_p)$ does not divide $2$ in $\Z(\zeta_p)$. If it did, then $N_{\Q(\zeta_p)/\Q}(1 - \zeta_p)$ should divide $N_{\Q(\zeta_p)/\Q}(2)$ in $\Z$, leading to a contradiction since $N_{\Q(\zeta_p)/\Q}(1 - \zeta_p) = p$ and $N_{\Q(\zeta_p)/\Q}(2) = 2^{p-1}$.
\end{remark}

\begin{remark}\label{remark2}
Let $D$ be a Dedekind Domain and $P = \langle a \rangle \neq 0$ a prime principal ideal of $D$. For any non-zero element $d \in D$, there exists an integer $k = 0,1, \ldots $ so that $a^k$ divides $d$. This follows from the unique factorization of the ideal $\langle d \rangle$ into a product of finite powers of prime ideals; that is, $a^k \mid d$ if and only if $P^k$ is a factor in the prime factorization of $\langle d \rangle$ and thus $0 \leq k < \infty$. 
\end{remark}

For any ring  $R$ and  any polynomial $g(x) \in R[x]$ 
we denote by $|\Supp(g(x))|$ the number of nonzero coefficients of $g(x)$. 
The following is Lemma 2 in \cite{Frenkel} and  for completeness we include its proof here.  
\begin{lemma}[Frenkel]\label{Lem:Frenkel}
Let  $ \F$ be a finite field of characteristic $p$ and    $0 \neq g(x) \in  \F[x]$  be a polynomial of degree $<p$. If $0 \neq a \in \F$ is a root of $g(x) $ with multiplicity $t$ then 
\[
t< |\Supp(g(x))|.
\]
\end{lemma}

\begin{proof}
We induct on the degree of $g(x)$,  with the base case, that of  constant polynomials being trivially true. 
Assume now the lemma holds for all polynomials of degree $<k$, for some fixed $1 \leq k <p$, and take $g(x)$ of degree $k$. If $g(0)= 0 $, then $|\Supp(g(x))|= |\Supp(g(x)/x)|$  and  $a$ is also a root of $g(x)/x$ with multiplicity $t$. By induction the lemma holds for $g(x)/x$ and thus for $g(x)$. So we may assume that $g(0) \neq 0 $. Then $|\Supp(g'(x))| = |\Supp(g(x))|-1$ and $a$ is a root of $g'(x)$ with multiplicity $\geq t-1$.
Because $g(x)$ is of positive degree $k < p$, its derivative  $g'(x) \neq 0$. So the lemma holds for $g'(x)$ and therefore also for $g(x)$. 
\end{proof}

\section{The case $n=2p$ }\label{Sec:n=2p}

Throughout this  section $n=2p$ with  $p$ being an odd prime. Observe that   $\zeta_n= (-1) \cdot \zeta_p$ and $\Z[\zeta_n]=\Z[\zeta_p]$. 

\begin{proofofA}
  The theorem is equivalent to saying that if numbers $z_j \in \Q(\zeta_n)$ $(j \in J)$ satisfy $\sum_{j \in J}z_j \zeta_n^{ij}=0$ for every $i \in I$, then all $z_j$ must be zero. Factoring out the  denominators, we can assume that $z_j \in  \Z[\zeta_p]$.  The ideal $P= \langle 1 - \zeta_p \rangle $ is prime in the Dedekind Domain $\Z[\zeta_p]$ and thus \autoref{remark2} implies that
  the maximum power of $1 - \zeta_p$ dividing $z_j $ is some  $k_j\in \N$. Dividing, if necessary, with $(1-\zeta_p)^{\min\{k_j\}}$ we may assume that there is at least one $z_j$ that is not divisible by $ 1 - \zeta_p$.
Thus we obtain the polynomial 
\begin{equation}\label{eq:s(x)}
r(x) = \sum_{j \in J }z_jx^{j}  \in \Z[\zeta_p][x] 
\end{equation}
which vanishes at $\zeta_n^i$ for all $i \in I$, and is such that  not all $z_j$ are divisible by $1-\zeta_p$.

Using the isomorphism and  notation from \eqref{eq:isomo} and \eqref{eq:zn^i} with $r=2$, we get $\zeta_{2p}^i= \zeta_p^{i_p}  \cdot (-1)^{i_2}$, 
for every $i \in I$.

Next we define a polynomial in two variables 
\begin{equation}\label{eq:g(x)}
g(x,y) = \sum_{j \in J }z_jx^{j_p}y^{j_2}  \in \Z[\zeta_p][x,y]. 
\end{equation}
We clearly have 
\begin{equation}\label{eq:rootsOfg(x)}
g(\zeta_p^{2i_p}, (-1)^{i_2}) = \sum_{j \in J }z_j \zeta_p^{2i_pj_p} (-1)^{i_2j_2} = \sum_{j \in J }z_j \zeta_n^{ij}= 0 
\end{equation}
for all $i \in I$.

Consider the sets 
\begin{equation}\label{eq:I^k}
  J_2^k = \{ j \in J \mid j_2 = k  \},    
\end{equation}
for $k=0,1$.

{\bf Case 1. } Assume first that $J = J_2^1$ and thus $J_2^0 = \emptyset$. Consequently, we have $I = I_2^1$ and $I_2^0 = \emptyset$. Thus  the polynomial 
\[
T(x):= g(x,-1) = - \sum_{j \in J }z_jx^{j_p}  \in \Z[\zeta_p][x]
\]
has  roots 
$T(\zeta_p^{2i_p})=0$,  for all $i \in I$. So $T(x)$ is divisible by $\prod_{i \in I}(x- \zeta_p^{2i_p})$. Observe that all elements $\{ i_p \mid i \in I_2^1\}$  are distinct in $\Z_p$ and thus all $\zeta_p^{2i_p}$ are also distinct $p$-th roots of unity (as $\zeta_p^2$ is also a primitive $p$-th root).
Now we pass to the quotient $\Z[\zeta_p] /\langle 1 - \zeta_p \rangle $ by applying the homomorphism $\Z[\zeta_p] \to \Z[\zeta_p]/\langle 1 - \zeta_p \rangle \cong  \F_p $  to the coefficients of $T(x)$ to get the polynomial  $\bar{T}(x) \in \F_p[x]$. Clearly $(x- \bar{1})^{|I|}$ divides $\bar{T}(x)$. As 
 $|\Supp(T(x))| \leq |J|$  and $|I|=|J|$, \autoref{Lem:Frenkel} implies that $\bar{T}(x) =  \bar{0}$ in $ \Z[\zeta_p]/\langle 1 - \zeta_p \rangle $. Therefore, all coefficients $z_j$ of $T(x)$ are divisible by $1-\zeta_p$, contradicting the initial selection of $z_j$ in $r(x)$. We conclude that Case 1 is impossible.
 
The case $J = J_2^0$ similarly results in a contradiction, so we move on to the next case.

{\bf Case 2. }
Neither $J_2^0$ nor $J_2^1$ is an empty set.

Consider the polynomials 
\begin{align}
T_0(x)&:=g(x,1)= \sum_{j \in J}  z_j x^{j_p}= S_0(x) +S_1(x)\\
T_1(x)&:=g(x,-1)= \sum_{j \in J}  z_j x^{j_p} (-1)^{j_2}=S_0(x) - S_1(x), 
\end{align}
where $S_0(x):= \sum_{j \in J_2^0}  z_j x^{j_p}$ and $S_1(x):= \sum_{j \in J_2^1}  z_j x^{j_p}$. 
In view of equation \eqref{eq:rootsOfg(x)},  
$T_0(\zeta_p^{2i_p}) = 0$ for all $i \in I_2^0$, 
and similarly $T_1(\zeta_p^{2i_p})= 0$ for all $i \in I_2^1$. As in Case 1, the elements $\{\zeta_p^{2i_p}\} $ are distinct for all $i \in I_2^0$,
and the same holds for $\{ \zeta_p^{2i_p} \mid i \in I_2^1\}$.    Hence $\prod_{i \in I_2^0}(x- \zeta_p^{2i_p})$ divides  $T_0(x)$ and $\prod_{i \in I_2^1}(x- \zeta_p^{2i_p})$ divides  $T_1(x)$.  Passing again to the quotient  $\Z[\zeta_p]/\langle 1 - \zeta_p \rangle $ and to the corresponding polynomials there  we conclude that 
\begin{align}
(x-\bar{1})^{|I_2^0|}\, \,  & \big| \, \,   \bar{T}_0(x) =  \bar{S}_0(x) + \bar{S}_1(x), \\ 
(x-\bar{1})^{|I_2^1|}\, \,  & \big| \, \,   \bar{T}_1(x) =  \bar{S}_0(x) - \bar{S}_1(x).
\end{align}
Without loss assume $|I_2^0 | \leq |I_2^1|$. Then 
$(x-\bar{1})^{|I_2^0|}$ divides $\bar{T}_0(x) +\bar{T}_1(x)= 2 \bar{S}_0(x)$ while $|\Supp(2\bar{S}_0(x))|  \leq |J_2^0|= |I_2^0|$. 
Applying \autoref{Lem:Frenkel}, we find that $2 \bar{S}_0(x) = \bar{0}$ in $\Z[\zeta_p]/\langle 1 - \zeta_p \rangle$. Since $2$ is not divisible by $1 - \zeta_p$ (as noted in \autoref{remark1}), it follows that $1 - \zeta_p$ divides $z_j$ for all $j \in J_2^0$. This implies $\bar{S}_0(x) = 0$ and that $\bar{T}_1(x) = \bar{S}_1(x)$ has $|\Supp(\bar{T}_1(x))| \leq |J_2^1| = |I_2^1|$. However, $(x - \bar{1})^{|I_j^1|}$ divides $\bar{T}_1$, which forces $\bar{T}_1(x) = \bar{0}$ in $\Z[\zeta_p]/\langle 1 - \zeta_p \rangle$ by \autoref{Lem:Frenkel}. Consequently, $1 - \zeta_p$ divides $z_j$ for all $j \in J_2^1$. Since $J = J_2^0 \cup J_2^1$, we conclude that $1 - \zeta_p$ divides $z_j$ for all $j \in J$, contradicting the choice of $r(x)$. 

The proof of the theorem is now complete.
\end{proofofA}

\section{ The case $n=pr $  }\label{Sec:n=pr}

We begin this section  with the definition of $\Gamma_r$ as provided in Section 3 of \cite{Zhang}; interested readers can refer to \cite{Zhang} for further details and motivation.
Denote by $V_n(x_1, x_2, \cdots, x_n)$ the  determinant of the $n\times n$ Vandermonde matrix whose $i, j$-entry is given as $x_i^{j-1}$,  for $1 \leq i, j \leq n$. So, 
\[
V_n(x_1, \cdots, x_n)= \prod_{1 \leq i \leq j \leq n} (x_i - x_j).
\]
For any odd prime $r$ and any  fixed $2 \leq n \leq r-1$, let 
\[
\gamma_n:= \max \Bigl\{ \frac{V_n(a_1, a_2, \cdots, a_n)}{V_n(0,1,\cdots,n-1) } \Big| \, 0 \leq a_1 <a_2 < \cdots <a_n \leq r-1 \Bigr\}.
\]
Then 
\[
\Gamma_r := \max \big\{ \gamma_n \, \big|  \, \, 2 \leq n \leq r-1 \big\}.
\]
Note that $\Gamma_r$ grows at least exponentially on $r$. If we choose $n=\frac{r-1}{2}$ and work with $a_i = 2i$,  for $i=0,1, \cdots  , n-1$,  
it is easy to see that $\Gamma_r$ is at least $2^{n \choose 2}$.
Nevertheless, as it was already noted in \cite{Zhang} (Remark 3.3)   for every prime $r$ there are infinitely many primes $p$ that are primitive in $\Z_r$.

\begin{proofofB}
As in the case of $n=2p$, we need to show that if the numbers $z_j \in \Q(\zeta_n)$ for $j \in J$ satisfy $\sum_{j \in J} z_j \zeta_n^{ij} = 0$ for every $i \in I$, then all $z_j$ must be zero. We can assume without loss of generality that $z_j \in \Z[\zeta_n]$. Note that $\langle 1 - \zeta_p \rangle$ is a prime ideal in $\Z[\zeta_n]$ (by \autoref{Lem:MaximalIdeal}), allowing us to apply \autoref{remark2} in $\Z[\zeta_n]$. Consequently, we obtain the polynomial 
\begin{equation}\label{eq:r(x)-n}
r(x) = \sum_{j \in J} z_j x^{j} \in \Z[\zeta_n][x]
\end{equation}
which vanishes at $\zeta_n^i$ for all $i \in I$, and is such that not all coefficients $z_j$ are  divisible by $1 - \zeta_p$.

Using the notation from \eqref{eq:isomo} and \eqref{eq:zn^i}, we define a polynomial in two variables 
\begin{equation}\label{eq:g(x)-n}
g(x,y) = \sum_{j \in J }z_jx^{j_p}y^{j_r}  \in \Z[\zeta_n][x,y]. 
\end{equation}
We clearly have 
\begin{equation}\label{eq:rootsOfg(x)-n}
g(\zeta_p^{ri_p},\zeta_r^{pi_r}) = \sum_{j \in J }z_j \zeta_p^{ri_pj_p}\zeta_r^{pi_rj_r} = \sum_{j \in J }z_j \zeta_n^{ij}= 0 
\end{equation}
for all $i \in I$.

Consider the sets 
\begin{equation}
  J_r^k = \{ j \in J \mid j_r = k  \},    
\end{equation}
and let  $L \subseteq \{ 0, 1, \cdots, r-1\}$ consisting  of those integers $k$ with $J_r^k \neq \emptyset$. 
By assumption, $|J_r^k|= |I_r^k|$ for all $k \in \{0, 1, \cdots, r-1\}$, so $L$ also identifies  the integers $k$ with $I_r^k \neq \emptyset$. 
Write $L = \{ k_1, k_2,  \cdots, k_{|L|} \}$ so that 
$|I_r^{k_1}| \leq |I_r^{k_2}| \leq \cdots \leq |I_r^{k_{|L|}}|$.

We define the polynomials  
\[
T_t(x) := g(x, \zeta_r^{pt}) \in \Z[\zeta_n][x],
\]  
for each  $t \in L$. Then  
\[
T_t(x) = \sum_{j \in J} z_j  \zeta_r^{ptj_r} x^{j_p} = \sum_{k \in L} \zeta_r^{ptk} \sum_{j \in J_r^k} z_j x^{j_p}. 
\]  
In view of \eqref{eq:rootsOfg(x)-n} we get $
T_t(\zeta_p^{ri_p})= 0$, 
for all $i \in I_r^t$ and thus 
\begin{equation}\label{eq:dividesT}
\prod_{i \in I_r^t} (x- \zeta_p^{ri_p}) \, \big|  \, T_t(x), 
\end{equation}
for every $t  \in L$. Observe that all elements $\{ i_p | i \in I_r^t \}$ are distinct in $\Z_p$, which means $\zeta_p^{ri_p}$ are also distinct $p$-th roots of unity, as $\zeta_p^r$ is a primitive $p$-th root of unity.

For every $k \in L$,  we consider the polynomials    
$S_k(x) := \sum_{j \in J_r^k} z_j x^{j_p}$ of $  \Z[\zeta_n][x]$ and express $T_t(x)$,  as  
\begin{equation}\label{eq:T-S}
T_t(x) = \sum_{k\in L}  \zeta_r^{ptk} S_k(x). 
\end{equation}
This way we have produced the $|L| \times |L|$ system 
\begin{equation}\label{eq:system}
   \left( \begin{matrix}
    T_1(x)\\
    \vdots\\
    T_{|L|}(x)\\
\end{matrix}
\right)
= R \cdot  \left( \begin{matrix}
    S_1(x)\\
    \vdots\\
    S_{|L|}(x)\\
\end{matrix}
\right) 
\end{equation}
with matrix $R= (R_{t,k})=(\zeta_r^{ptk})_{t, k \in L}= (\omega_r^{tk})_{t, k \in L}$,  where $\omega_r= \zeta_r^p$ is also a primitive $r$-th root of unity. 

Now we pass to the quotient $\Z[\zeta_n]/\langle 1- \zeta_p \rangle \cong \F_{p^{r-1}}$. 
Applying the homomorphism 
$\Z[\zeta_n] \to \Z[\zeta_n]/\langle 1- \zeta_p \rangle \cong \F_{p^{r-1}}$ to the coefficients of  all the polynomials involved we get $\bar{T}_t(x)$  and $\bar{S}_k(x)  \in \F_{p^{r-1}}[x]$ satisfying the system 
\begin{equation}\label{eq:systemInF_p}
   \left( \begin{matrix}
    \bar{T}_1(x)\\
    \vdots\\
    \bar{T}_{|L|}(x)\\
\end{matrix}
\right)
= \bar{R} \cdot  \left( \begin{matrix}
    \bar{S}_1(x)\\
    \vdots\\
    \bar{S}_{|L|}(x)\\
\end{matrix}
\right) 
\end{equation}
with matrix $\bar{R}= (\bar{R}_{t,k})=(\bar{\omega}_r^{tk})_{t, k \in L}$.
Observe also that 
\begin{equation}\label{eq:bar{s}}
    \bar{S}_k(x)=\sum_{j \in J_r^k} \bar{z}_j x^{j_p}. 
\end{equation}

Based on \autoref{Lem:z_p-primitive}, the element $\bar{\omega}_r$ in $\Z[\zeta_n]/\langle 1- \zeta_p \rangle $ is a primitive $r$-th root of unity. Therefore, $\bar{R}$ is a square submatrix of $(\bar{\omega}_r^{ij})_{i,j = 0}^{r-1}$ and all the  hypothesis of \autoref{Zhang} are satisfied. Consequently, $\bar{R}$ is invertible and so 
\begin{equation*}
   \left( \begin{matrix}
    \bar{S}_1(x)\\
    \vdots\\
    \bar{S}_{|L|}(x)\\
\end{matrix}
\right)
= \bar{R}^{-1} \cdot  \left( \begin{matrix}
    \bar{T}_1(x)\\
    \vdots\\
    \bar{T}_{|L|}(x)\\
\end{matrix}
\right). 
\end{equation*}
In particular, 
\begin{equation}\label{eq:S_k1}
\bar{S}_{k_1}(x) = \sum_{t \in L} \bar{a}_t \bar{T}_t(x), 
\end{equation}
where $\bar{a}_t \in \Z[\zeta_n]/ \langle 1- \zeta_p \rangle$.  
According to \eqref{eq:dividesT}, $\bar{T}_t(x)$ is divisible by $(x-\bar{1})^{|I_r^t|}$, which means that $(x-\bar{1})^{|I_r^{k_1}|}$ divides $\bar{T}_t(x)$ for every $t \in L$, given that $|I_r^{k_1}|$ is the minimum in the set $\{ |I_r^k| \}_{k \in L}$. Consequently, $(x-\bar{1})^{|I_r^{k_1}|}$ also divides $\bar{S}_{k_1}(x)$. Furthermore, equation \eqref{eq:bar{s}} implies that $|\Supp(\bar{S}_{k_1}(x))| \leq |J_r^{k_1}| = |I_r^{k_1}|$. Using  Frenkel's Lemma, we conclude that $\bar{S}_{k_1}(x) = \bar{0}$. Thus, $z_j$ is divisible by $1-\zeta_p$ for all $j \in J_r^{k_1}$.

Since we have established that $\bar{S}_{k_1}(x)=\bar{0}$, we can deduce from the system \eqref{eq:systemInF_p} that
\begin{equation}
   \left( \begin{matrix}
    \bar{T}_{k_2}(x)\\
    \vdots\\
    \bar{T}_{k_{|L|}}(x)\\
\end{matrix}
\right)
= \bar{D} \cdot  \left( \begin{matrix}
    \bar{S}_{k_2}(x)\\
    \vdots\\
    \bar{S}_{k_{|L|}}(x)\\
\end{matrix}
\right) 
\end{equation}
where $\bar{D}= (\bar{D}_{k_t,k_l})=(\bar{\omega}_r^{k_tk_l})_{k_t, k_l \in L'}$  for $L'= L \setminus \{ k_1 \}$. All hypotheses of \autoref{Zhang} still hold, ensuring that the matrix $\bar{D}$ is invertible. Repeating the previous argument yields $\bar{S}_{k_2} = \bar{0}$, and thus   $z_j$ is divisible by $1 - \zeta_p$ for all $j \in J_r^{k_2}$. By continuing this process and reducing the matrix dimensions by one each time, we conclude that $z_j$ is divisible by $1 - \zeta_p$ for all $j \in \bigcup_{i=1}^{|L|} J_r^{k_i} = J$. This clearly contradicts the choice of $r(x)$ and the proof of the 
theorem is complete.
\end{proofofB}

\begin{remark}
    In \cite{Zhang}, Examples 4.1 and 4.3 demonstrate the necessity of the second condition in \autoref{Zhang}; however, the resulting singular submatrices are not principal.
\end{remark}

\begin{remark}
    Explicit examples of primes that meet the criteria of \autoref{Zhang} and \autoref{Thm:B} were also presented in \cite{Zhang}. For instance, 
    $\Gamma_3= 2$, $\Gamma_5=8$  and $\Gamma_7=75$. This means that, for example,  \autoref{Thm:B} is applicable for $n=3\cdot 5$  (as $5$ is primitive in $\Z_3$ and greater than $\Gamma_3$) but not for $n= 3 \cdot 7$ because on one hand the order of $7$ in $\Z_3^*$ is $1$  and  on the other hand  $3$ is not greater than $\Gamma_7$ 
    although  $3$ is primitive in $\Z_7$.
\end{remark}

\section{An uncertainty principle}\label{Sec:Uncertainty}
In this section we give a generalization of Tao's uncertainty principle. His proof  (see \cite{Tao})
 was based on the following observation: Assume that $G =\Z_n$ is a cyclic group of order $n$ and let 
\( f: G \to \C \) be   a non-zero function. 
If \( |\supp(f)| + |\supp(\widehat{f} \, )| \leq n \), then  
\[ |\supp(f)| \leq n - |\supp(\widehat{f} \, )| = |\{ \lambda \in \widehat{G} \mid \widehat{f} (\lambda) = 0 \}|. \]
Thus, if $ A = \supp(f)\subseteq G=  \Z_p$, there exists a subset $ B \subseteq \widehat G= \Z_p $
such that $|A| = |B| $ and $ \widehat{f}(b) = 0 $ for all $ b \in B $. In particular, the linear map $T: l^2(A) \to l^2(B)$  defined by $T: g \to \widehat{g} |_B $ is singular as $T(f)=0$ and  $f\neq 0$ 
(by  $l^2(A)$ is denoted the set of functions $g: G \to \C$  which are 
$0$ out of $A$).
Observe now that the matrix of $T$ is precisely the submatrix $(\zeta_n^{ji})_{i\in A, j \in B}$ of $(\zeta_n^{kl})_{0 \leq k , l \leq n-1}$. 
Using the same  observation we can show \autoref{Prop:Equivalence} below.

(The notation used is as in the introduction. In particular we write any element of $\Z_r \times \Z_m$ as an integer modulo $rm$ using  the isomorphism $\Z_r \times \Z_m \cong \Z_{rm} $ that maps any vector  $(i_r, i_m)$  to $i\equiv i_rm +i_mr  \pmod {rm}$ and we write   $G_i= \{i\} \times \Z_m$ , for all $i=0, \cdots , r-1$.)

\begin{theorem}\label{Prop:Equivalence}
Let $r, m$ be positive integers with  $(r, m)=1$.
\begin{itemize}
\item[(a)]
Suppose $f : G=\Z_r \times \Z_m \to \C$  is  a non-zero function   satisfying  
\begin{equation}\label{eq:layeredf}
|\supp(f) \cap G_k| + |\supp(\widehat{f} \, ) \cap \widehat{G}_k| \leq m,  \ 
  \text{ for all } k=0,1, \cdots ,r-1.
\end{equation}
If the sets $I \subseteq G$,$J \subseteq \widehat{G}$ satisfy
\begin{equation}\label{eq:layeredIJ}
  |I \cap G_k|= |J \cap \widehat{G}_k|,\ \text{ for all } k=0,1, \cdots ,r-1,  
\end{equation}
and $\supp(f) \subseteq I$, $\widehat{f}|_J =0$ then the matrix $(\zeta_{rm}^{ji})_{i \in I, j \in J}$ is singular. There is at least one such pair $I \subseteq G$,$J \subseteq \widehat{G}$,  for each such $f$.
\item[(b)]
Conversely, if  $I \subseteq G =\Z_r \times \Z_m$ and $J \subseteq \widehat{G} = \Z_r \times \Z_m$  satisfy  \eqref{eq:layeredIJ} 
and the matrix $(\zeta_{rm}^{ji})_{i \in I, j \in J}$ is singular then there exists $f : G=\Z_r \times \Z_m \to \C$  with $\supp(f) \subseteq I$ and $\widehat{f}|_J = 0$ satisfying  \eqref{eq:layeredf}.
\end{itemize}

\end{theorem}

Observe that if  $r=1$  and $m=p$   in \autoref{Prop:Equivalence} we get  the  cyclic $p$-group   $G= \{1\} \times \Z_p$, for which Chebotarev's theorem holds (i.e. all the square submatrices $(\zeta_p^{ij})$ are nonsingular), and  thus we recover Tao's uncertainty principle.

\begin{proof} (a)
Assume first that  $f : G=\Z_r \times \Z_m \to \C$ is a non-zero function satisfying \eqref{eq:layeredf}. If $I, J$ satisfy the hypothesis in (a), then the matrix $(\zeta_{rm}^{ji})_{i \in I, j \in J}$ defines a linear map $T: l^2(I) \to l^2(J)$ that maps any $g \in l^2(I)$ to $\widehat{g}|_J$. Since $T(f) = \widehat{f} |_J = 0$ for the non-zero function $f$ (as $\supp(f) \subseteq I$), the matrix $(\zeta_{rm}^{ji})_{i \in I, j \in J}$ is singular.

To complete the proof of (a) it remains  to show that, for the given $f$,  such a pair $I, J$ 
really exists. For every  $k =0, 1, \cdots , r-1$, define $I_k = \supp(f) \cap G_k$. The fact that 
\[
|\supp(f) \cap G_k| + |\supp(\widehat{f} \, ) \cap \widehat{G}_k| \leq m 
\]
implies that there exist set $J_k \subseteq \widehat{G}_k$ with $|I_k| = |J_k|$ such that $\widehat f (b) = 0$ for all $b \in J_k$. Let $I = \bigcup_{k=0}^{r-1} I_k $ and $J = \bigcup_{k=0}^{r-1} J_k$. Clearly $I = \supp{f}$, $\widehat{f}|_J=0$ and $I, J$ satisfy \eqref{eq:layeredIJ}.

(b) Conversely, suppose $(\zeta_{rm}^{ji})_{i \in I, j \in J}$ is a singular matrix, where $I \subseteq G$ and $ J  \subseteq  \widehat{G}$   satisfy \eqref{eq:layeredIJ}.  The singularity of the matrix implies that the corresponding linear map $T$ maps a nonzero function to zero; that is, there exists a nonzero function $f: \Z_r \times \Z_m \to \C$ such that $\supp(f)  \subseteq I$ and $T( f)  = \widehat{f} |_J= 0$.
Let   $k  \in \{  0,1, \cdots ,r-1\} $, then 
\[
|\supp(f) \cap G_k| \leq  |I \cap G_k |, 
\]
since  $\supp(f)  \subseteq I$. Also 
\[
|\supp(\widehat{f}) \cap \widehat{G}_k| \leq m- |J \cap G_k| = m - |I \cap G_k|, 
\]
where the first inequality follows from the fact that $ |\widehat{G}_k|= m$ and  $\widehat{f}|_{J}= 0$, while the last equality follows  by hypothesis. 
Adding the above we get $|\supp(f) \cap G_k| + |\supp(\widehat{f} \, ) \cap \widehat{G}_k| \leq m$ for all $k=0, 1, \cdots, r-1$.

This completes the proof of the theorem.
\end{proof}

\autoref{Prop:Unc-A} (and similarly \autoref{Prop:Unc-B}) follows directly from \autoref{Prop:Equivalence} and \autoref{Thm:A} (respectively, \autoref{Thm:B}). This is because:

\begin{proofofPropA}
    By  \autoref{Thm:A}, for any $I \subseteq Z_{2p},  J \subseteq \widehat{\Z_{2p}}$ that satisfy \eqref{eq:layeredIJ}, the  matrix $(\zeta_{2p}^{ji})_{i \in I, j\in J}$  is non singular. Hence by the  first part of \autoref{Prop:Equivalence} no non-zero $f:\Z_{2p} \to \C$ satisfies \eqref{eq:layeredf}. The proposition now follows.
\end{proofofPropA}

The proof of \autoref{Prop:Unc-B} is similar.

\section{Complementary matrices} 
\label{Complementary}

Let $A$ be an $n \times n $ matrix and write $[n]$ for the set $\{ 1, 2, \cdots n \}$.  For every $I \subseteq [n]$ we denote by $I^c$ the complementary subset $[n] \setminus I $. If $I, J \subseteq [n]$ with $|I|=|J|$  we write  $A_{I, J}$ for the $|I| \times |I|$ submatrix of $A$ obtained from $A$ by removing all rows whose indices do not belong to $I$ and all columns  whose indices do not belong to $J$. Observe that
$A_{I,J} = A^t_{J,I}$  for all $I, J$.  If $I = J$ we simply write $A_I$ and this is a principal submatrix of $A$.

 It is known,  that a principal submatrix  of a matrix $A$  is non-singular  if and only if its complementary principal submatrix  is also non-singular(see for example Proposition 5.4 in \cite{Brow} for a more general result). We present here a simple proof of this fact based on 
Jacobi's complementary minor theorem, which states:
\begin{theorem}[Jacobi]
   Assume $A $ is an invertible $ n \times n$ matrix over a field $K$ and let $I, J \subseteq [n]$ with $|I|=|J|$. Then 
   \[
   \det (A_{I,J}) = (-1)^{\sum I +\sum J} \cdot \det A \cdot \det( (A^{-1})_{J^c, I^c}).
   \]
\end{theorem}
With the use of the adjoint $\adj A$ of A we have  $A^{-1} = \frac{1}{\det A} \cdot \adj A$ and  so  
\[
(A^{-1})_{J^c, I^c}=  \frac{1}{\det A} \cdot (\adj A)_{J^c, I^c}= \frac{1}{\det A} \cdot ((\adj A)^t)_{I^c, J^c}
\]
Hence Jacobi's formula is translated to 
\begin{equation}\label{eq:jacobi-adj}
   \det (A_{I,J})= (-1)^{\sum I +\sum J} \cdot \frac{\det(A)}{\det(A)^{|I^c|}} \cdot \det(((\adj A)^t)_{I^c, J^c}).
\end{equation}

We can now prove 
\begin{proposition}\label{prop:complements}
    Let $A =(\omega^{k,l})_{0 \leq k, l \leq n-1} $ where $\omega$ a primitive $n$-th root of unity. For any  $I, J  \subseteq [n]$ with $|I|=|J|$ the submatrix $A_{I, J}$ is invertible if and only if $A_{I^c,J^c} $ is invertible.  In particular, for  $I=J$ we have that the principal minor $\det{A_I}$ is non zero if and only if the principal minor $\det{A_{I^c}}$ is non zero. 
\end{proposition}

\begin{proof}
The matrix $A$ is the character table of the cyclic group $C_n$ of order $n$ and as such it is invertible and satisfies     
\[
A \cdot \bar{A}^t= n \cdot I_n. 
\]
This along with the formula for the adjoint $\adj A$ of $A$ implies that  $\det(A)  \cdot \bar{A} = n \cdot (\adj A)^t$ and thus 
\[
 det(A)  \cdot (\bar{A})_{I^c,J^c} = n \cdot ((\adj A)^t)_{I^c,J^c}   
\]
for any $I, J  \subseteq [n]$ with $|I|=|J|$. If  $k = n-|I|$, taking  determinants to the above equation we get  
\begin{equation}\label{eq:Det-adj}
(\det(A))^k \cdot \det((\bar{A})_{I^c,J^c}) = n^k \cdot \det(((\adj A)^t)_{I^c,J^c})
\end{equation}
Clearly $\det((\bar{A})_{I^c,J^c})=\overline{ \det(A_{I^c,J^c})} $ and  we substitute to Jacobi's formula  \eqref{eq:jacobi-adj} to  get   
\[
\det (A_{I,J})=(-1)^{\sum I +\sum J}\cdot \frac{\det(A)}{(\det(A))^{k}} \cdot \frac{(\det(A))^k}{n^k} \cdot \overline{ \det(A_{I^c,J^c})} =
(-1)^{\sum I +\sum J} \cdot  \frac{\det(A)}{n^k} \cdot \overline{ \det(A_{I^c,J^c})}.
\] 
Hence $\det (A_{I,J}) \neq 0 $  if and only if $\det(A_{I^c,J^c})) \neq 0 $  and the proposition follows.

\end{proof}

{\bf Acknowledgment of priority }
Proposition 3.6 in \cite{Goetz} contains a similar proof of \autoref{prop:complements} above. I thank G. Pfander for pointing out this reference.

\end{document}